\documentclass[reqno]{amsart}
\usepackage[dvips]{color}
\usepackage{epic}
\usepackage{comment}
\usepackage{relsize}
\usepackage{dsfont}
\usepackage{tikz}
\usetikzlibrary{matrix}
\usepackage[T1]{fontenc}
\usepackage{faktor}
\usepackage[utf8]{inputenc}
\usepackage{amsmath}
\usepackage{xfrac}
\usepackage{xcolor}
\usepackage{mathtools}
\usepackage{amsfonts}
\usepackage{amssymb}
\usepackage{amsthm}
\usepackage{mathrsfs}
\usepackage[below]{placeins}
\def\Xint#1{\mathchoice
	{\XXint\displaystyle\textstyle{#1}}%
	{\XXint\textstyle\scriptstyle{#1}}%
	{\XXint\scriptstyle\scriptscriptstyle{#1}}%
	{\XXint\scriptscriptstyle\scriptscriptstyle{#1}}%
	\!\int}
\def\XXint#1#2#3{{\setbox0=\hbox{$#1{#2#3}{\int}$}
		\vcenter{\hbox{$#2#3$}}\kern-.5\wd0}}

\def\dashint{\Xint-}

\newcommand{\ddu}{\frac{d}{du}_{|u=0}}
\newcommand{\Z}{\mathbb{Z}}

\newcommand{\N}{\mathbb{N}}

\newcommand{\Q}{\mathbb{Q}}
\newcommand{\R}{\mathbb{R}}

\newcommand{\DDs}{\frac{\partial}{\partial s}}

\newcommand{\Cinf}{\mathcal{C}^{\infty}}
\newcommand{\Tr}{\mathbb{T}}

\newcommand{\Chi}{\Large\raisebox{2pt}{$\chi$}}

\newcommand{\vect}{\mathrm{span}}
\newcommand{\Lie}{\mathrm{Lie}}

\newcommand{\Diff}{\mathrm{Diff}}
\newcommand{\Aut}{\mathrm{Aut}}

\newcommand{\supp}{\mathrm{supp}}

\newcommand{\g}{\mathfrak{g}}

\def \rat{ {\rm Q}\kern-.65em {}^{{}_/ }}

\newtheorem{Theorem}{Theorem}[subsection]
\newtheorem{Propo}{Proposition}[subsection]
\newtheorem{Corol}{Corollary}[subsection]

\newtheorem{Lemma}{Lemma}[subsection]

\newtheorem{Remark}{Remark}[subsection]

\input{tcilatex}

\title{A compact manifold with infinite-dimensional co-invariant cohomology}

\begin{document}

\author{Mehdi Nabil}                 

\address{%
	M. Nabil\\               
	Cadi Ayyad University, Faculty of Sciences Semlalia, Department of Mathematics, Marrakesh. Morocco\\            
	mehdi1nabil@gmail.com                
}

\date{\today}
\maketitle

{\noindent\bf Abstract.} Let $M$ be a smooth manifold. When $\Gamma$ is a group acting on $M$ by diffeomorphisms one can define the $\Gamma$-co-invariant cohomology of $M$ to be the cohomology of the complex $\Omega_c(M)_\Gamma=\vect\{\omega-\gamma^*\omega,\;\omega\in\Omega_c(M),\;\gamma\in\Gamma\}$. For a Lie algebra $\mathcal{G}$ acting on the manifold $M$, one defines the cohomology of $\mathcal{G}$-divergence forms to be the cohomology of the complex $\mathcal{C}_{\mathcal{G}}(M)=\vect\{L_X\omega,\;\omega\in\Omega_c(M),\;X\in\mathcal{G}\}$. In this short paper we present a situation where these two cohomologies are infinite dimensional. 
\vskip 0.2cm

\noindent{\bf Mathematics Subject Classification 2010:} 57S15, 14F40.\\
{\bf Keywords:}  Cohomology, Transformation Groups.

\vspace{1cm}

\maketitle 

\section{Introduction}\label{section1}
In \cite{ABN}, the authors have introduced the concept of co-invariant cohomology. In basic terms it is the cohomology of a subcomplex of the de Rham complex generated by the action of a group on a smooth manifold. The authors showed that under nice enough hypotheses on the nature of the action, there is an interplay between the de Rham cohomology of the manifold, the cohomology of invariant forms and the co-invariant cohomology, and this relationship can be exhibited either by vector space decompositions or through long exact sequences depending on the case of study (Theorems $1.1$ and $1.3$ in \cite{ABN}). Among the various consequences that can be derived from this inspection, it is evident that the dimension of de Rham cohomology has some control over the dimension of the co-invariant cohomology, and in most cases presented in \cite{ABN} the latter is finite whenever the former is. This occurs for instance in the case of a finite action on a compact manifold or more generally in the case of an isometric action on a compact oriented Riemannian manifold, and this fact holds as well for a non-compact manifold as long as one requires the action to be free and properly discontinuous with compact orbit space. A concept closely related to co-invariant cohomology is the cohomology of divergence forms, which is defined by means of a Lie algebra action on a smooth manifold and was introduced by A. Abouqateb in \cite{Abouqateb}. In the course of his study, the author gave many examples where the cohomology of divergence forms is finite-dimensional.

The goal of this paper is to show that this phenomenon heavily depends on the nature of the action in play, and that without underlying hypotheses, co-invariant cohomology and cohomology of divergence forms are not generally well-behaved. This is illustrated by an example of a vector field action on a smooth compact manifold giving rise to infinite-dimensional cohomology of divergence forms and whose discrete flow induces an infinite-dimensional co-invariant cohomology as opposed to the de Rham cohomology of the manifold. This shows in particular that many results obtained in \cite{ABN} and \cite{Abouqateb} cannot be easily generalized and brings into perspective the necessity to look for finer finiteness conditions of co-invariant cohomology in a future study which would put the present paper in a broader context.\\

The general outline of the paper is as follows : In the first paragraph, we briefly recall the notions of co-invariant forms and divergence forms, then we define an homomorphism of the de Rham complex that is induced by a complete vector field on the manifold, and which maps divergence forms relative to the action of the vector field onto the complex of co-invariant differential forms associated to its discrete flow (see \eqref{opcodiv} and Proposition \ref{propcodiv}). The next paragraph is concerned with the setting on which our cohomology computations will take place, it comprises a smooth compact manifold, the $3$-dimensional hyperbolic torus, which can be obtained as the quotient of a solvable Lie group by a uniform lattice (the construction given here is that of A. EL Kacimi in \cite{KacimiManuscripta}), the Lie algebra action considered is by means of a left-invariant vector field. We then use a number of results to prove Theorem \ref{thcodiv} which states that the operator defined in \eqref{opcodiv} is an isomorphism between the complex of divergence forms and the complex of co-invariant forms, hence allowing to only consider the cohomology of co-invariant forms for computation. Finally, the last paragraph is dedicated to the main computation in which we prove that the discrete flow of the vector field in question on the hyperbolic torus gives infinite-dimensional co-invariant cohomology.
\subsection*{Acknowledgement} The author would like to thank Abdelhak Abouqateb for his helpful discussions and advice concerning this paper. 
\newpage
\section{Preliminaries}
Let $M$ be a smooth $n$-dimensional manifold and denote $\Diff(M)$ the group of diffeomorphisms of $M$ and $\Chi(M)$ the Lie algebra of smooth vector fields on $M$. Let $\rho:\Gamma\longrightarrow\Diff(M)$ be an action of a group $\Gamma$ on $M$ by diffeomorphisms. For an $r$-form $\omega$ on $M$ and element $\gamma\in\Gamma$, we denote $\gamma^*\omega$ the pull-back of $\omega$ by the diffeomorphism $\rho(\gamma):M\longrightarrow M$. Let $\Omega_c(M)=\oplus_p\Omega_c^p(M)$ denote the de Rham complex of forms with compact support on $M$ and put:
$$\Omega_c^p(M)_\rho:=\vect\{\omega-\gamma^*\omega,\;\gamma\in\Gamma,\;\omega\in\Omega_c^p(M)\}.$$
Any element of $\Omega_c^p(M)_\rho$ is called a $\rho$-co-invariant or just a ($\Gamma$-)co-invariant when there is no ambiguity. The graded vector space $\Omega_c(M)_\rho:=\oplus_p\Omega_c^p(M)_\rho$ is a differential subcomplex of the de Rham complex $\Omega_c(M)$, it is called the complex of co-invariant differential forms on $M$. When $M$ is compact this complex is simply denoted $\Omega(M)_\rho$. In the case where $\rho:\Z\longrightarrow M$ is the action induced by a diffeomorphism $\gamma:M\longrightarrow M$, i.e $\rho(n):=\gamma^n$, then we get that:
$$\Omega^p_c(M)=\{\omega-\gamma^*\omega,\;\omega\in\Omega^p_c(M)\}.$$
Let $\tau:\mathcal{G}\longrightarrow \Chi(M)$ be a Lie algebra homomorphism and denote $\hat{X}:=\tau(X)$ for any $X\in\mathcal{G}$ then define:
$$\mathcal{C}^p_\tau(M):=\vect\{L_{\hat{X}}\omega,\;X\in\mathcal{G},\;\omega\in\Omega^p_c(M)\}.$$
Any element of $\mathcal{C}^p_\tau(M)$ is called a $\tau$-divergence $p$-form or simply $\mathcal{G}$-divergence form. The graded vector space $\mathcal{C}_\tau(M):=\oplus_p\mathcal{C}^p_\tau(M)$ is a differential subcomplex of the de Rham complex. If $X$ is any vector field on $M$, with corresponding Lie algebra homomorphism $\tau:\R\longrightarrow\Chi(M)$, $\tau(1):=X$ then:
$$\mathcal{C}^p_\tau(M)=\vect\{L_X\omega,\;\omega\in\Omega^p_c(M)\}.$$
In what follows, $X\in\Chi(M)$ is a complete vector field and $\phi:M\times[0,1]\longrightarrow M$ is the flow $\phi^X$ of the vector field $X$ restricted to $M\times[0,1]$. We define the linear operator $I:\Omega(M)\longrightarrow\Omega(M)$ by the expression:
\begin{equation}
\label{opcodiv}
I(\eta):=\dashint_0^1 \phi^*\eta\wedge\mathrm{pr}_2^*(ds)
\end{equation}
where $\Large\raisebox{2pt}{$\dashint_0^1$}:\Omega^*(M\times[0,1])\longrightarrow\Omega^{*-1}(M)$ is the fiberwise integration operator of the trivial bundle $M\times[0,1]\overset{\mathrm{pr}_1}{\longrightarrow} M$ (see \cite{G.H.V.2}) and $ds$ the usual volume form on $[0,1]$.\\ Let $\tau:\R\longrightarrow \Chi(M)$ be the Lie algebra homomorphism induced by $X$ and let $\rho:\Z\longrightarrow\Diff(M)$ be the discrete flow of $X$ i.e the group action given by $\rho(n):=\phi_n^X$.
\begin{Propo}
\label{propcodiv}
The operator $I:\Omega(M)\longrightarrow\Omega(M)$ defined by $\eqref{opcodiv}$ is a differential complex homomorphism i.e $I\circ d=-d\circ I$. Moreover $I(\mathcal{C}_\tau(M))\subset \Omega_c(M)_\rho$ and the restriction of $I:\mathcal{C}_\tau(M)\longrightarrow \Omega_c(M)_\rho$ is surjective.
\end{Propo}
\begin{proof}
	Let $\eta\in\Omega(M)$ and denote $\iota_s:M\longrightarrow M\times\{s\}\hookrightarrow M\times[0,1]$ be the natural inclusion, then using Stokes formula for fiberwise integration we get that:
	$$ I(d\eta)=\dashint_0^1 \phi^*(d\eta)\wedge \mathrm{pr}_2^*(ds)=\dashint_0^1 d(\phi^*\eta\wedge \mathrm{pr}_2^*(ds))=-dI(\eta)+\big[\iota_s^*(\phi^*\eta\wedge \mathrm{pr}_2^*ds)\big]_0^1,$$
	and since $\iota_s^*\mathrm{pr}_2^*(ds)=0$ then $I(d\eta)=-dI(\eta)$. For the second claim we start by showing that $I(\eta)$ has compact support whenever $\eta$ does. Indeed assume $\eta\in\Omega_c(M)$ and denote $K:=\supp(\eta)$, next consider the map:
	$$ f:M\times\R\longrightarrow M,\;\;\; (x,s)\mapsto\phi_s^{-1}(x):=\phi(x,-s),$$
	Then $f$ is continuous and therefore $L:=f(K\times[0,1])$ is compact. For any $y\in M\setminus L$ and any $s\in [0,1]$ we get that $\phi_s(y)\notin K$ and therefore $(\phi^*\eta)_{(y,s)}=0$, this implies that $I(\eta)_y=0$. We conclude that $\supp\;I(\eta)\subset L$ i.e $I(\eta)\in\Omega_c(M)$.\\
	From the relation $\mathrm{T}_{\small(x,t)}\phi(0,1)=X_{\phi_t(x)}$ one gets that $\phi^*\circ i_X=i_{(0,\DDs)}\circ\phi^*$ and therefore $\phi^*\circ L_X=L_{(0,\DDs)}\circ\phi^*$. Moreover we have that $$L_{(0,\DDs)}\mathrm{pr}_2^*(ds)=0\;\;\;\text{and}\;\;\;\mathlarger\dashint_0^1\circ i_{(0,\DDs)}=0.$$ 
	If we write $\eta=L_X\omega$ for some $\omega\in\Omega_c(M)$ then we get that:
	\begin{eqnarray*}
		I(L_X\omega)&=&\dashint_0^1\phi^*(L_X\omega)\wedge \mathrm{pr}_2^*(ds)\\
		&=&\dashint_0^1L_{(0,\DDs)}(\phi^*\omega)\wedge \mathrm{pr}_2^*(ds)\\
		&=&\dashint_0^1L_{(0,\DDs)}(\phi^*\omega\wedge \mathrm{pr}_2^*(ds))\\
		&=&\dashint_0^1d\circ i_{(0,\DDs)}(\phi^*\omega\wedge \mathrm{pr}_2^*(ds))+\dashint_0^1i_{(0,\DDs)} d(\phi^*\omega\wedge \mathrm{pr}_2^*(ds))\\
		&=& d\left(\dashint_0^1 i_{(0,\DDs)}(\phi^*\omega\wedge \mathrm{pr}_2^*(ds))\right)+[\iota_s^*i_{(0,\DDs)}(\phi^*\omega\wedge \mathrm{pr}_2^*(ds))]_0^1\\[.1in]
		&=& [\phi_s^*\omega]_0^1\\[.1in]
		&=& \phi_1^*\omega-\phi_0^*\omega\\[.1in]
		&=&\phi_1^*\omega-\omega.
	\end{eqnarray*}
	It follows that $I(\mathcal{C}_\tau(M))\subset\Omega_c(M)_\rho$. This also shows that $I:\mathcal{C}_\tau(M)\longrightarrow\Omega_c(M)_\rho$ is surjective.
\end{proof}
\begin{Remark}
Note that $\phi^X$-invariant forms on $M$ are fixed by $I$ i.e if $\omega\in\Omega(M)$ such that $L_X\omega=0$ then $I(\omega)=\omega$.
\end{Remark}
\section{The hyperbolic torus}
\noindent Consider $A\in \mathrm{SL}(2,\Z)$ with $\mathrm{tr}(A)>2$. It is easy to check that $A=PDP^{-1}$ for some $P\in\mathrm{GL}(2,\R)$ and $D=\mathrm{diag}(\lambda,\lambda^{-1})$. Clearly $\lambda>0$ and $\lambda\neq 1$. Hence it makes sense to set $D^t=\mathrm{diag}(\lambda^t,\lambda^{-t})$ and define $A^t=PD^tP^{-1}$ for any $t\in\R$. Next we define the Lie group homomorphism:
$$\phi:\R\longrightarrow\Aut(\R^2),\;\;\;t\mapsto A^t$$
The hyperbolic torus $\mathbb{T}^3_A$ is the smooth manifold defined as the quotient
$\Gamma_3\backslash G_3$ where $G_3:=\R^2\rtimes_\phi \R$ and $\Gamma_3:=\Z^2\rtimes_\phi\Z$. The natural projection $\R^2\rtimes_\phi\R\overset{p}\longrightarrow \R$ induces a fiber bundle structure $\Tr^3_A\overset{p}{\longrightarrow}\mathbb{S}^1$ with fiber type $\mathbb{T}^2$ and $p[x,y,t]=[t]$.\\$\ $\\ If $(1,a)$ and $(1,b)$ are the eigenvectors of $A$ respectively associated to the eigenvalues $\lambda$ and $\lambda^{-1}$ then:
$$ v=(1,a,0),\;\;\;w=(1,b,0)\;\;\;\text{and}\;\;\;e=(0,0,-\log(\lambda)^{-1}),$$
forms a basis of $\g_3=\Lie(G_3)$, and we can check that:
\begin{equation}
\label{brackethyp}
[v,w]_{\g_3}=0,\;\;\;[e,v]_{\g_3}=-v,\;\;\;\text{and}\;\;\;[e,w]_{\g_3}=w.
\end{equation}

\noindent Denote $X$, $Y$ and $Z$ the left invariant vector fields on $\R^2\rtimes_\phi\R$ associated to $v$, $w$ and $e$ respectively, then $\{X,Y,Z\}$ defines a parallelism on $\Tr_3^A$, a direct calculation leads to:
\begin{equation}
\label{vechyp}
X=\lambda^t\left(\frac{\partial}{\partial x}+a\frac{\partial}{\partial y}\right),\;\;\;Y=\lambda^{-t}\left(\frac{\partial}{\partial x}+b\frac{\partial}{\partial y}\right)\;\;\;\text{and}\;\;\; Z=-\log(\lambda)^{-1}\frac{\partial}{\partial t}.
\end{equation}
Now denote $\alpha$, $\beta$ and $\theta$ the dual forms associated to $X$, $Y$ and $Z$ respectively. It is clear that the vector fields $X$ and $Y$ of $\Tr^3_A$ are tangent to the fibers of the fiber bundle $\Tr^3_A\overset{p}{\longrightarrow} \mathbb{S}^1$, and that $\theta=-(\log\lambda)p^*(\sigma)$ where $\sigma$ is the invariant volume form on $\mathbb{S}^1$ satisfying $\int_{\mathbb{S}^1}\sigma=1$. Assume in what follows that the eigenvalue $\lambda$ of $A$ is irrational, then from the relation:
$$ A\left(\begin{matrix}
1\\a\end{matrix}\right)=\lambda\left(\begin{matrix}
1\\a\end{matrix}\right)$$
we deduce that $a\in\R\setminus\Q$. This remark leads to:
\begin{Propo}
	\label{propoXhyp}
	The orbits of the vector field $X$ defined in \eqref{vechyp} are dense in the fibers of the fiber bundle $\Tr^3_A\overset{p}{\longrightarrow} \mathbb{S}^1$. In particular for any $f\in\Cinf(\Tr^3_A)$, $X(f)=0$ is equivalent to $f=p^*(\phi)$ for some $\phi\in\Cinf(\mathbb{S}^1)$.
\end{Propo}
\begin{proof}
	We shall identify $\mathbb{S}^1$ with $\R/\Z$. Fix $[t]\in \mathbb{S}^1$ and consider the diffeomorphism:
	$$ \Phi_t:\Tr^2\longrightarrow p^{-1}[t],\;\;\;[x,y]\mapsto[x,y,t].$$
	Then define the vector field $\hat{X}$ on $\Tr^2$ given by: $$\hat{X}_{[x,y]}=T_{[x,y,t]}\phi_t^{-1}(X_{[x,y,t]})=\lambda^t\left(\frac{\partial}{\partial x}+a\frac{\partial}{\partial y}\right).$$ 
	Since $a$ is irrational we get that the family $\{1,a\}$ is $\Q$-linearly independent and thus the orbits of $\hat{X}$ are dense in $\Tr^2$, consequently the orbits of $X$ are dense in $p^{-1}[t]$, this proves the assertion since $t$ is arbitrary.
\end{proof}
The following Lemma is of central importance for the development of this paragraph and for the computations of the next section:
\begin{Lemma}
	\label{lemmadensedis}
	Let $f\in\Cinf(\mathbb{T}^3_A)$ then for every $s\in\R$ we have the following formula:
	\begin{equation}
	Z\big((\phi_s^X)^*(f)\big)=-s(\phi_s^X)^*\big(X(f)\big)+(\phi_s^X)^*\big(Z(f)\big).
	\end{equation}
	In particular $Z(\gamma^*f)=-X(\gamma^*f)+\gamma^*(Z(f))$ and $i_Z\circ\gamma^*=-\gamma^*\circ i_X +\gamma^*\circ i_Z$, where $\gamma:=\phi_1^X$.
\end{Lemma}
\begin{proof}
	For any $(x,y,t)\in\R^3$, a straightforward computation gives that:
	\begin{eqnarray*}
		Z\big((\phi_s^X)^*(f)\big)(x,y,t)&=&-\frac{1}{\log\lambda}d(f\circ\phi_s^X)_{(x,y,t)}(0,0,1)\\
		&=&-\frac{1}{\log\lambda}\ddu (f\circ\phi_s^X)(x,y,t+u)\\
		&=&-\frac{1}{\log\lambda}\ddu f(s\lambda^{t+u}+x,as\lambda^{t+u}+y,t+u)\\
		&=&-\frac{1}{\log\lambda} (df)_{\phi_s^X(x,y,t)}(s\log(\lambda)\lambda^t,as\log(\lambda)\lambda^t,1)\\
		&=&-s(df)_{\phi_s^X(x,y,t)}(\lambda^t,a\lambda^t,0)-\frac{1}{\log\lambda}(df)_{\phi_s^X(x,y,t)}(0,0,1)\\
		&=&-s (X(f)\circ\phi_s^X)(x,y,t)+(Z(f)\circ\phi_s^X)(x,y,t).
	\end{eqnarray*}
	Which achieves the proof.
\end{proof}
\begin{Corol}
	\label{coroldendis1}
	Let $f\in\Cinf(\mathbb{T}^3_A)$ such that $f=\gamma^*f$ with $\gamma:=\phi_1^X$. Then $X(f)=0$ and consequently $f=p^*\psi$ with $\psi\in\Cinf(\mathbb{S}^1)$.
\end{Corol}
\begin{proof}
	Since $f=\gamma^*f$ we get that for every $n\in\Z$, $f=(\gamma^n)^*(f)$ thus the preceding lemma gives that:
	$$ Z(f)=-nX(f)+(\gamma^n)^*(Z(f)).$$
	Consequently we obtain that for every $n\in\Z$:
	$$|X(f)|\leq \frac{1}{n} \big(\lVert Z(f)\lVert_{\infty}+\lVert (\gamma^n)^*(Z(f))\lVert_{\infty}\big)\leq \frac{2}{n}\lVert Z(f)\lVert_{\infty},$$
	which leads to $X(f)=0$ and achieves the proof.
\end{proof}
\noindent In what follows we denote $M:=\Tr^3_A$. It is straightforward to check that:
$$d\alpha=-\alpha\wedge\theta,\;\;\;d\beta=\beta\wedge\theta,\;\;\;d\theta=0.$$
and that $L_X\alpha=-\theta$ and $L_X\beta=L_X\theta=0$ thus $L_X(\alpha\wedge\beta\wedge\theta)=0$.\\
Let $\tau:\R\longrightarrow\Chi(M)$ be the Lie algebra homomorphism corresponding to the vector field $X$ and  $\rho:\Z\longrightarrow \Diff(M)$ the discrete action  generated by $\gamma:=\phi_1^X$ where $\phi^X$ is the flow of $X$, that is, $\rho(n)(x)=\phi_n^X(x)$ for any $n\in \Z$.
\begin{Theorem}
\label{thcodiv}
The homomorphism $I:\mathcal{C}_\tau(M)\longrightarrow \Omega(M)_\rho$ defined in \eqref{opcodiv} is an isomorphism.
\end{Theorem}
\begin{proof}
	In view of Proposition \ref{propcodiv} it only remains to prove that $I$ is injective. Choose $\eta\in\mathcal{C}_\tau(M)$ and write $\eta=L_X\omega$ for some $\omega\in\Omega(M)$. Assume that $I(\eta)=0$, in view of the previous computation this is equivalent to $\omega=\gamma^*\omega$.\\\\
	If $\eta\in\mathcal{C}_\tau^0(M)$ then Corollary \ref{coroldendis1} gives that $\omega=p^*\phi$ for some $\phi\in\Cinf(\mathbb{S}^1)$, thus $\eta=0$. On the other hand if $\eta\in\mathcal{C}_\tau^3(M)$ we can write $\eta=X(f)\alpha\wedge\beta\wedge\theta$ for some $f\in\Cinf(M)$ satisfying $f=\gamma^*f$ and so by Corollary \ref{coroldendis1} we get $\eta=0$.\\\\ Now for $\eta\in\mathcal{C}_\tau^1(M)$ we can write:
	$$ \omega=f\alpha+g\beta+h\theta,\;\;\;\;\;f,g,h\in\Cinf(M).$$
	Applying $I$ to $L_X\alpha=-\theta$ leads to $\gamma^*\alpha=\alpha-\theta$. Moreover since $\theta$ and $\beta$ are $\phi^X$-invariant we get that $\beta=\gamma^*\beta$ and $\theta=\gamma^*\theta$ thus:
	$$ \gamma^*\omega=(\gamma^*f)\alpha+(\gamma^*g)\beta+(\gamma^*h-\gamma^*f)\theta,$$
	hence $\omega=\gamma^*\omega$ is equivalent to $f=\gamma^*f$, $g=\gamma^*g$ and $h=\gamma^*h-f$. The last relation then implies that $h-(\gamma^n)^*h=nf$ for all $n\in\Z$ and thus:
	$$\lVert f\lVert_\infty\leq \frac{1}{n}\lVert h-(\gamma^n)^*h\lVert_\infty\leq\frac{2}{n}\lVert h\lVert_\infty\underset{n\rightarrow+\infty}{\longrightarrow} 0.$$
	Therefore $f=0$ and $h=\gamma^*h$, $g=\gamma^*g$ which according to Corollary \ref{coroldendis1} gives that $X(g)=0$ and $X(h)=0$, and so using that $L_X\beta=0$ and $L_X\theta=0$ it follows that $\eta=L_X\omega=0$. Finally let $\eta\in\mathcal{C}^2_\tau(M)$ and write:
	$$ \omega=f\alpha\wedge\beta+g\alpha\wedge\theta+h\beta\wedge\theta,\;\;\;\;\;f,g,h\in\Cinf(M).$$
	Then using $\gamma^*\alpha=\alpha-\theta$ we obtain that:
	$$\gamma^*\omega=(\gamma^*f)\alpha\wedge\beta+(\gamma^*g)\alpha\wedge\theta+(\gamma^*h+\gamma^*f)\beta\wedge\theta,$$
	and so $\omega=\gamma^*\omega$ is equivalent in this case to $f=\gamma^*f$, $g=\gamma^*g$ and $h=\gamma^*h+\gamma^*f$. As before, this leads to $f=0$, $X(g)=0$ and $X(h)=0$ and so:
	$$ \eta=L_X\omega=L_X(g\alpha\wedge\theta+h\beta\wedge\theta)=g L_X(\alpha\wedge\theta)=-g\theta\wedge\theta=0.$$
	Thus $I:\mathcal{C}_\tau(M)\longrightarrow\Omega(M)_\rho$ is an isomorphism.
\end{proof}
This result gives in particular that $\mathrm{H}(\mathcal{C}_\tau(M))\simeq\mathrm{H}(\Omega(M)_\rho)$ and therefore we only need to compute the cohomology of $\rho$-co-invariant forms in this case.
\section{Cohomology computation}
We now have all the necessary ingredients to perform our computation. Let $M$ denote the hyperbolic torus $\mathbb{T}^3_A$ defined in the previous section with $A$ having irrational eigenvalues and let $X,Y,Z\in\Chi(M)$ be the vector fields defined in \eqref{vechyp} with respective dual $1$-forms $\alpha$, $\beta$ and $\theta$. Define the action $\rho:\Z\longrightarrow\Diff(M)$ to be the discrete flow of the vector field $X$ with $\gamma:=\rho(1)$. The main goal is to prove that first and second co-invariant cohomology groups are infinite-dimension, however we shall compute the whole cohomology in order get a global picture.\\\\
\noindent\textsc{Calculating} $\mathrm{H}^0(\Omega(M)_\rho)$: Choose $f\in\Omega^0(M)_\rho$ such that $df=0$, then $f$ is a constant function equal to $g-\gamma^*g$ for some $g\in\Cinf(M)$. Consequently we obtain that:
$$\int_{M}f\alpha\wedge\beta\wedge\theta=\int_{M}(g-\gamma^*g)\alpha\wedge\beta\wedge\theta=\int_{M}g\alpha\wedge\beta\wedge\theta-\int_{M}\gamma^*(g\alpha\wedge\beta\wedge\theta)=0.$$
Thus $f=0$ and we conclude that $\mathrm{H}^0(\Omega(M)_\rho)=0$.\\\\
\textsc{Calculating} $\mathrm{H}^1(\Omega(M)_\rho)$: We prove that $\mathrm{H}^1(\Omega(M)_\rho)$ is infinite dimensional. In order to do so, we prove that the map $p^*:\Omega^1(\mathbb{S}^1)\longrightarrow\mathrm{H}^1(\Omega(M)_\rho)$ is well-defined and injective or equivalently we can show that $p^*(\Omega^1(\mathbb{S}^1))\subset \mathrm{Z}^1(\Omega(M)_\rho)$ and $p^*(\Omega^1(\mathbb{S}^1))\cap \mathrm{B}^1(\Omega(M)_\rho)=0$.\\\\ An element $\eta\in p^*(\Omega^1(\mathbb{S}^1))$ can always be written as $\eta=p^*(\phi)\theta$ where $\phi\in\Cinf(\mathbb{S}^1)$. Since $L_X\theta=0$ and $L_X\alpha=-\theta$, then by applying $I$ to $L_X\alpha$ we get that $\theta=\alpha-\gamma^*\alpha$, therefore: 
$$\eta=p^*(\phi)\theta=p^*(\phi)\alpha-\gamma^*(p^*(\phi)\alpha).$$
Moreover observe that $d\eta=0$, hence we deduce that $p^*(\Omega^1(\mathbb{S}^1))\subset \mathrm{Z}^1(\Omega(M)_\rho)$. Now suppose $\eta=d(g-\gamma^*g)$ then clearly $X(g-\gamma^*g)=0$ and $Z(g-\gamma^*g)=p^*(\phi)$, thus according to Proposition \ref{propoXhyp}, $g-\gamma^*g=p^*\psi$ for some $\psi\in\Cinf(\mathbb{S}^1)$. By induction we can show that for any $n\in\N$, $g=\rho(n)^*g+np^*\psi$ which then leads to:
\begin{equation}
\label{modeleq}
|p^*\psi|\leq \frac{1}{n}|g-\rho(n)^*g|\leq\frac{2}{n}\lVert g\lVert_\infty\underset{n\rightarrow+\infty}{\longrightarrow}0.
\end{equation}
Hence $p^*\psi=g-\gamma^*g=0$ and so $\eta=0$. Thus $p^*(\Omega^1(\mathbb{S}^1))\cap \mathrm{B}^1(\Omega(M)_\Z)=0$.\\\\
\textsc{Calculating} $\mathrm{H}^2(\Omega(M)_\rho)$: We will show that $p^*(\Omega^1(\mathbb{S}^1))\wedge\beta\subset\mathrm{H}^2(\Omega(M)_\rho)$. To do this, we fix a $2$-form $\eta=p^*(\phi)\theta\wedge\beta$ such that $\phi\in\Cinf(\mathbb{S}^1)$. We can easily check that $d\eta=0$, moreover from the previous calculations and the fact that $L_X\beta=0$ we get that $\beta=\gamma^*\beta$, therefore:
$$p^*(\phi)\theta\wedge\beta=(p^*\phi\alpha\wedge\beta)-\gamma^*(p^*(\phi)\alpha\wedge\beta).$$
Hence $p^*(\Omega^1(\mathbb{S}^1))\wedge\beta\subset\mathrm{Z}^2(\Omega(M)_\rho)$. Now assume that $\eta=d(\omega-\gamma^*\omega)$, then using that $[X,Y]=0$ we get $i_Yi_X(d\omega)=\gamma^*(i_Yi_Xd\omega)$ hence according to Corollary \ref{coroldendis1} we can write $i_Xi_Y d\omega=p^*\psi$ for some $\psi\in\Cinf(\mathbb{S}^1)$. On the other hand we get from  Lemma \ref{lemmadensedis} that: $$p^*\phi-i_Yi_X(d\omega)=\gamma^*(i_Zi_Yd\omega)-i_Zi_Yd\omega.$$
It follows from these remarks that $p^*(\phi-\psi)=\gamma^*(i_Zi_Yd\omega)-i_Zi_Yd\omega$, and as in \eqref{modeleq} we once again prove that $p^*(\phi-\psi)=0$, so we deduce that $p^*\phi=p^*\psi=i_Yi_X(d\omega)$. Now if we write $\omega=f\alpha+g\beta+h\theta$, then we get that $p^*\phi=X(g)-Y(f)$. Moreover, from  $X(p^*\phi)=Y(p^*\phi)=0$ we get that for every $s\in\R$:
\begin{eqnarray*}
	s^2p^*\phi&=&\int_0^s\int_0^s(\phi_t^X)^*(\phi_u^Y)^*(p^*\phi)dudt\\
	&=&\int_0^s\int_0^s(\phi_t^X)^*(\phi_u^Y)^*(X(g))dudt-\int_0^s\int_0^s(\phi_t^X)^*(\phi_u^Y)^*(Y(f))dudt\\
	&=&\int_0^s(\phi_t^X)^*X\left(\int_0^s(\phi_u^Y)^*(g)du\right)dt-\int_0^s(\phi_u^Y)^*Y\left(\int_0^s(\phi_t^X)^*(f)dt\right)du\\
	&=&(\phi_s^X)^*\left(\int_0^s(\phi_u^Y)^*(g)du\right)-\int_0^s(\phi_u^Y)^*(g)du-(\phi_s^Y)^*\left(\int_0^s(\phi_t^X)^*(f)dt\right)+\int_0^s(\phi_t^X)^*(f)dt.
\end{eqnarray*}
It follows that:
$$ s^2|p^*\phi|\leq 2\left\lVert\int_0^s(\phi_u^Y)^*(g)du\right\lVert_\infty+2\left\lVert\int_0^s(\phi_t^X)^*(f)dt\right\lVert_\infty\leq 2|s|(\lVert g\lVert_\infty+\lVert f\lVert_\infty)$$
Hence $|p^*\phi|\leq \dfrac{2}{|s|}(\lVert g\lVert_\infty+\lVert f\lVert_\infty)\underset{s\rightarrow+\infty}{\longrightarrow}0$.\\
We conclude that $\eta=0$ and $p^*(\Omega^1(\mathbb{S}^1))\wedge\beta\cap\mathrm{B}^2(\Omega(M)_\rho)=0$, in particular this proves that $\mathrm{H}^2(\Omega(M)_\rho)$ is infinite dimensional.\\\\
\textsc{Calculating} $\mathrm{H}^3(\Omega(M)_\rho)$: The elements of $\Omega^3(M)_\rho$ are of the form: 
$$(f-\gamma^*f)\alpha\wedge\beta\wedge\theta,$$
for some $f\in\Cinf(M)$. Put:
$$c=\dfrac{\int_M f\alpha\wedge\beta\wedge\theta}{\alpha\wedge\beta\wedge\theta},\;\;\;\text{then}\;\int_M (f-c)\alpha\wedge\beta\wedge\theta=0.$$
Thus $(f-c)\alpha\wedge\beta\wedge\theta=d\omega$ and since $L_X(\alpha\wedge\beta\wedge\theta)=0$ then $(\gamma^*f-c)\alpha\wedge\beta\wedge\theta=d(\gamma^*\omega)$ and therefore it follows that: $$(f-\gamma^*f)\alpha\wedge\beta\wedge\theta=d(\omega-\gamma^*\omega),$$
i.e $\mathrm{H}^3(\Omega(M)_\rho)=0$.


\newpage
\begin{thebibliography}{99}
\bibitem{ABN} A. Abouqateb, M. Boucetta and M. Nabil, {\em Cohomology of Coinvariant Differential Forms}. Journal of Lie Theory 28, no. 3 (2018): 829-841.\\
\bibitem{Abouqateb} A. Abouqateb, {\em Cohomologie des formes divergences  et Actions propres d'alg\`ebres de Lie.} Journal of Lie Theory 17 (2007), No. 2, 317-335.\\
\bibitem{KacimiManuscripta} A. El Kacimi Alaoui. \newblock {\em Invariants de certaines actions de Lie instabilité du caractère Fredholm}, \newblock Manuscripta Mathematica 74.1 (1992): 143-160.\\
\bibitem{G.H.V.2}  W. Greub, S. Halperin and R. Vanstone, \newblock{\em Connections,
	Curvature and Cohomology,} \newblock {Vol. II}, Academic Press 1972/1973.
	
	
	
\end{thebibliography}
\end{document}